\documentclass[a4paper,12pt]{amsart}

\usepackage{amsmath}
\usepackage{amssymb}
\usepackage{mathrsfs}
\usepackage{enumerate}
\usepackage{ifthen}
\usepackage{graphicx}
\usepackage[T1]{fontenc} 
\usepackage{color}

\nonstopmode \numberwithin{equation}{section}
\newtheorem{thm}{Theorem}[section]
\newtheorem{cor}{Corollary}[section]
\newtheorem{lem}{Lemma}[section]
\newtheorem{prop}{Proposition}[section]

\newtheorem{conj}{Conjecture}

\theoremstyle{definition}

\newtheorem{prob}{Problem}[section]
\newtheorem{rem}{Remark}[section]

\newtheorem{innercustomthm}{Theorem}
\newenvironment{customthm}[1]
  {\renewcommand\theinnercustomthm{#1}\innercustomthm}
  {\endinnercustomthm}

\newcounter{minutes}
\divide\time by 60
\newcounter{hours}
\multiply\time by 60
\addtocounter{minutes}{-\time}

\newcounter {own}
\def\theown {\thesection       .\arabic{own}}

\newenvironment{pf}[1][]{%
 \vskip 3mm
 \noindent
 \ifthenelse{\equal{#1}{}}%
  {{\slshape Proof. }}%
  {{\slshape #1.} }%
 }%
{\qed\bigskip}

\newcounter{alphabet}

\def\be{\begin{equation}}
\def\ee{\end{equation}}

\newcommand{\bee}{\begin{enumerate}}
\newcommand{\eee}{\end{enumerate}}

\newcommand{\blem}{\begin{lem}}
\newcommand{\elem}{\end{lem}}
\newcommand{\bthm}{\begin{thm}}
\newcommand{\ethm}{\end{thm}}
\newcommand{\bcor}{\begin{cor}}
\newcommand{\ecor}{\end{cor}}
\newcommand{\beg}{\begin{examp}}
\newcommand{\eeg}{\end{examp}}
\newcommand{\begs}{\begin{examples}}
\newcommand{\eegs}{\end{examples}}
\newcommand{\bdefe}{\begin{defin}}
\newcommand{\edefe}{\end{defin}}
\newcommand{\bprob}{\begin{prob}}
\newcommand{\eprob}{\end{prob}}
\newcommand{\bei}{\begin{itemize}}
\newcommand{\eei}{\end{itemize}}

\newcommand{\bcon}{\begin{conj}}
\newcommand{\econ}{\end{conj}}
\newcommand{\bcons}{\begin{conjs}}
\newcommand{\econs}{\end{conjs}}
\newcommand{\bprop}{\begin{prop}}
\newcommand{\eprop}{\end{prop}}
\newcommand{\br}{\begin{rem}}
\newcommand{\er}{\end{rem}}
\newcommand{\brs}{\begin{rems}}
\newcommand{\ers}{\end{rems}}
\newcommand{\bo}{\begin{obser}}
\newcommand{\eo}{\end{obser}}
\newcommand{\bos}{\begin{obsers}}
\newcommand{\eos}{\end{obsers}}
\newcommand{\bpf}{\begin{pf}}
\newcommand{\epf}{\end{pf}}
\newcommand{\ba}{\begin{array}}
\newcommand{\ea}{\end{array}}
\newcommand{\beq}{\begin{eqnarray}}
\newcommand{\beqq}{\begin{eqnarray*}}
\newcommand{\eeq}{\end{eqnarray}}
\newcommand{\eeqq}{\end{eqnarray*}}

\begin{document}

\title{The Schwarzian norm estimates for Janowski convex functions}

\author{Md Firoz Ali}
\address{Md Firoz Ali,
Department of Mathematics, National Institute of Technology Durgapur,
Durgapur- 713209, West Bengal, India.}
\email{ali.firoz89@gmail.com, firoz.ali@maths.nitdgp.ac.in}

\author{Sanjit Pal}
\address{Sanjit Pal,
Department of Mathematics, National Institute of Technology Durgapur,
Durgapur- 713209, West Bengal, India.}
\email{palsanjit6@gmail.com}

\subjclass[2010]{Primary 30C45, 30C55}
\keywords{analytic functions, univalent functions, Janowski convex functions, Schwarzian norm}

\def\thefootnote{}
\footnotetext{ {\tiny File:~\jobname.tex,
printed: \number\year-\number\month-\number\day,
          \thehours.\ifnum\theminutes<10{0}\fi\theminutes }
} \makeatletter\def\thefootnote{\@arabic\c@footnote}\makeatother

\begin{abstract}
For $-1\leq B<A\leq 1$, let $\mathcal{C}(A,B)$ denote the class of normalized Janowski convex functions defined in the unit disk $\mathbb{D}:=\{z\in\mathbb{C}:|z|<1\}$ that satisfy the subordination relation $1+zf''(z)/f'(z)\prec (1+Az)/(1+Bz)$. In the present article, we determine the sharp estimate of the Schwarzian norm for functions in the class $\mathcal{C}(A,B)$. The Dieudonn\'{e}'s lemma which gives the exact region of variability for derivatives at a point of bounded functions, plays the key role in this study, and we also use this lemma to construct the extremal functions for the sharpness by a new method.
\end{abstract}

\thanks{}

\maketitle
\pagestyle{myheadings}
\markboth{Md Firoz Ali and Sanjit Pal}{The Schwarzian norm estimates for Janowski convex functions}

\section{Introduction}
Let $\mathcal{H}$ denote the class of analytic functions in the unit disk $\mathbb{D}:=\{z\in\mathbb{C}:|z|<1\}$ and  $\mathcal{LU}$ denote the subclass of $\mathcal{H}$ consisting of all locally univalent functions, $i.e.$, $\mathcal{LU}=\{f\in\mathcal{H}: f'(z)\ne 0,~\text{for all}~ z\in\mathbb{D}\}$. For a locally univalent function $f\in\mathcal{LU}$, the Schwarzian derivative is defined by
$$S_f(z)=\left[\frac{f''(z)}{f'(z)}\right]^{'}-\frac{1}{2}\left[\frac{f''(z)}{f'(z)}\right]^2,$$
and that of the Schwarzian norm (the hyperbolic sup-norm) is defined by
$$||S_f||=\sup\limits_{z\in\mathbb{D}}(1-|z|^2)^2|S_f(z)|.$$
In $1949$,  Nehari \cite{Nehari-1949} proved that for a locally univalent function $f\in\mathcal{LU}$ with $||S_f||\leq 2$, the function $f$ is univalent in $\mathbb{D}$. Moreover, for a univalent function $f$, it is well known that $||S_f||\leq 6$ (see \cite{Kraus-1936, Nehari-1949}). Both of the constants $2$ and $6$ are best possible.\\

In the theory of quasiconformal mappings and Teichm\"{u}ller spaces, the Schwarzian norm has a significant meaning (see \cite{Lehto-1987}). A mapping $f:\widehat{\mathbb{C}}\rightarrow\widehat{\mathbb{C}}$ of the Riemann sphere $\widehat{\mathbb{C}}=\mathbb{C}\cup\{\infty\}$ is said to be $k$-quasiconformal ($0\le k<1$) mapping if $f$ is a sense preserving homeomorphism of $\widehat{\mathbb{C}}$ and has locally integrable partial derivatives on $\mathbb{C}\setminus\{f^{-1}(\infty)\}$ with $|f_{\bar{z}}|\le k|f_z|$ a.e.. A set of Schwarzian derivatives of analytic and univalent functions on $\mathbb{D}$ with quasiconformal extensions to $\widehat{\mathbb{C}}$ can be used to identify the theory of Teichm\"{u}ler space $\mathcal{T}$. In the Banach space $\mathcal{H}$, of all analytic functions in $\mathbb{D}$, it is known that $\mathcal{T}$ is a bounded domain with respect to finite hyperbolic sup-norm (see \cite{Lehto-1987}). The following theorem establishes a connection between the Schwarzian norm and the quasiconformal mapping.

\begin{customthm}{A}\label{Thm-r-0001}\cite{Ahlfors-Weill-1962, Kuhnau-1971}
If $f$ extends to a $k$-quasiconformal ($0\le k<1$) mapping of the Riemann sphere $\widehat{\mathbb{C}}$ then $||S_f||\le 6k$. Conversely, if $||S_f||\le 2k$ then $f$ extends to a $k$-quasiconformal mapping of the Riemann sphere $\widehat{\mathbb{C}}$.
\end{customthm}
Although the fundamental work on the Schwarzian derivative in connection with the theory of geometric functions have been done in \cite{Ahlfors-Weill-1962, Kuhnau-1971, Nehari-1949}), limited work has been done on the Schwarzian derivative for various subclasses of univalent functions.  Estimating the Schwarzian norm for typical subclasses of univalent functions in relation to Teichm\"{u}ller spaces is an interesting problem.\\

Let $\mathcal{A}$ denote the class of functions $f$ in $\mathcal{H}$ normalized by $f(0)=0$, $f'(0)=1$. Thus, a function $f$ in $\mathcal{A}$ has the Taylor series expansion of the form
\begin{equation}\label{r-00001}
f(z)=z+\sum\limits_{n=2}^{\infty}a_nz^n.
\end{equation}
Let $\mathcal{S}$ be the set of all functions $f\in\mathcal{A}$ that are univalent in $\mathbb{D}$. A function $f\in\mathcal{A}$ is called starlike (respectively, convex) if the image $f(\mathbb{D})$ is a starlike domain with respect to the origin (respectively, convex domain). The classes of all univalent starlike and convex functions are denoted by $\mathcal{S}^*$ and $\mathcal{C}$, respectively. It is well known that a function $f\in\mathcal{A}$ is starlike (respectively, convex) if and only if ${\rm Re\,} [zf'(z)/f(z)] > 0$  (respectively, ${\rm Re\,} [1 + zf''(z)/f'(z)] > 0$) for $z\in \mathbb{D}$. A function $f\in\mathcal{A}$ is said to be starlike (respectively, convex) of order $\alpha$, $0\le \alpha< 1$ if ${\rm Re\,} [zf'(z)/f(z)] > \alpha$ (respectively, ${\rm Re\,} [1 + zf''(z)/f'(z)] > \alpha$) for $z\in \mathbb{D}$. The set of all starlike and  convex functions of order $\alpha$ are denoted by  $\mathcal{S}^*(\alpha)$ and $\mathcal{C} (\alpha)$, respectively. See \cite{Duren-1983,Goodman-book-1983} for further information on these classes.\\

For the class of convex functions $\mathcal{C}$, the Schwarzian norm satisfies $||S_f||\le 2$ and the estimate is sharp. This result was proved repeatedly by many researchers (see \cite{Lehto-1977, Nehari-1976, Robertson-1969}). In 1996, Suita \cite{Suita-1996} studied the class $\mathcal{C}(\alpha)$, $0\le \alpha\le 1$ and using integral representation for functions in  $\mathcal{C}(\alpha)$ proved that the Schwarzian norm satisfies the following sharp inequality
$$||S_f||\le
\begin{cases}
2 & \text{ if }~ 0\le \alpha\le 1/2,\\
8\alpha(1-\alpha) & \text{ if }~ 1/2\le \alpha\le 1.
\end{cases}
$$
A function $f\in\mathcal{A}$ is called strongly starlike (respectively, strongly convex) of order $\alpha$, $0<\alpha<1$ if  $|\arg\{zf'(z)/f(z)\}|<\pi\alpha/2$ (respectively, $|\arg\{1+zf''(z)/f'(z)\}|<\pi\alpha/2$) for $z\in \mathbb{D}$. The classes of strongly starlike and strongly convex functions of order $\alpha$ are denoted by $\mathcal{S^*_{\alpha}}$ and $\mathcal{C}_{\alpha}$, respectively. For geometric properties on $\mathcal{S}^*_{\alpha}$, we refer to \cite{Kanas-2005, Kanas-Sugawa-2005}.
For $0<\alpha<1$, Fait {\it et al.} \cite{Fait-Krzyz-Zygmunt-1976} studied the class $\mathcal{S}^*_{\alpha}$ and  proved that a function $f\in\mathcal{S}^*_{\alpha}$ extends to an $\sin(\pi\alpha/2)$-quasiconformal mapping of $\widehat{\mathbb{C}}$. It is obvious from Theorem \ref{Thm-r-0001} that the norm satisfies $||S_f||\le 6\sin(\pi\alpha/2)$ which was pointed out by Chiang \cite{Chiang-1991}.
Kanas and Sugawa \cite{Kanas-Sugawa-2011} studied the Schwarzian norm for functions in the class $\mathcal{C}_{\alpha}$, $0<\alpha<1$ and proved that the sharp inequality $||S_f||\leq 2\alpha$ and therefore, $f$ extends to an $\alpha$-quasiconformal mapping of $\widehat{\mathbb{C}}$.\\

A function $f\in \mathcal{A}$ is said to be uniformly convex (see \cite{Goodman-1991, Ma-Minda-1992, Ronning-1993}) if and only if
$${\rm Re\,}\left(1+\frac{zf''(z)}{f'(z)}\right)> \left|\frac{zf''(z)}{f'(z)}\right|\quad\text{for}~ z\in\mathbb{D}.$$
Kanas and Sugawa \cite{Kanas-Sugawa-2011} proved the sharp inequality $||S_f||\le 8/\pi^2$ for $f\in\mathcal{UCV}$. They also proved that $f$ can be extended to a $4/\pi^2$-quasiconformal mapping of $\widehat{\mathbb{C}}$. In 2012, Bhowmik and Wirths \cite{Bhowmik-Wirths-2012} studied the class of concave functions $f\in\mathcal{A}$, with opening angle at infinity less than or equal to $\pi\alpha$, $\alpha\in [1,2]$, and obtained the sharp estimate $||S_f||\le 2(\alpha^2-1)$ and $f$ extends to an $(\alpha^2-1)$-quasiconformal mapping of $\widehat{\mathbb{C}}$ for $1\le \alpha<\sqrt{2}$. Recently, the present authors \cite{Ali-Sanjit-2022} considered two classes of functions $\mathcal{G}(\beta)$ with $\beta>0$ and $\mathcal{F}(\alpha)$ with $-\frac{1}{2}\le \alpha\le 0$, consisting of functions in $\mathcal{A}$ that satisfy the relation ${\rm Re\,}\left(1+\frac{zf''(z)}{f'(z)}\right)<1+\frac{\beta}{2}$ for $z\in\mathbb{D}$, and ${\rm Re\,}\left(1+\frac{zf''(z)}{f'(z)}\right)>\alpha$ for $z\in\mathbb{D}$, respectively,
and obtained the sharp estimates $||S_f||\le 2\beta(\beta+2)$ for $f\in\mathcal{G}(\beta)$ and $||S_f||\le 2(1-\alpha)/(1+\alpha)$ for $f\in\mathcal{F}(\alpha)$. Moreover, a function $f\in\mathcal{G}(\beta)$ can be extended to a $\beta(\beta+2)$-quasiconformal mapping of $\widehat{\mathbb{C}}$ for $0<\beta<\sqrt{2}-1$.\\

Let $f$ and $g$ be two analytic functions in $\mathbb{D}$. The function $f$ is said to be subordinate to $g$ if there exists an analytic function $\omega:\mathbb{D}\rightarrow\mathbb{D}$ with $\omega(0)=0$ such that $f(z)=g(\omega(z))$ and it is denoted by $f\prec g$. Moreover, when $g$ is univalent, $f\prec g$ if and only if $f(0)=g(0)$ and $f(\mathbb{D})\subset g(\mathbb{D})$. In this article, we consider the class of Janowski convex functions $f\in\mathcal{A}$ satisfying
$$1+\frac{zf''(z)}{f'(z)}\prec \frac{1+Az}{1+Bz},$$
where $-1\le B<A\le 1$. The class of all Janowski convex functions was first introduced and studied by Janowski \cite{Janowski-1973a, Janowski-1973b} and is denoted by $\mathcal{C}(A,B)$. It is easy to see that functions in $\mathcal{C}(A, B)$ are also convex functions. Silverman and Silvia \cite{Silverman-Silvia-1985} proved a necessary and sufficient conditions for the class $\mathcal{C}(A, B)$.
Kim and Sugawa \cite{Kim-Sugawa-2006}, and Ponnusamy and Sahoo \cite{Ponnusamy-Sahoo-2008} also studied the class $\mathcal{C}(A, B)$ and obtained the sharp bound of pre-Schwarzian norm. For particular values of $A$ and $B$, the class reduces to many well known classes. For example, the class $\mathcal{C}(1,-1)=:\mathcal{C}$ is the family of convex functions; for $0\leq\alpha<1$, $\mathcal{C}(1-2\alpha,-1)=:\mathcal{C}(\alpha)$ is the family of convex functions of order $\alpha$, introduced by Robertson \cite{Robertson-1936}.
For $-1\le B<A\le 1$,  define the function
\begin{equation}\label{r-00010}
K_{A,B}(z)=\begin{cases}
(1/A)\{(1+Bz)^{A/B}-1\} & \text{ if }~ A\neq 0,~B\neq 0,\\
(1/B)\log (1+Bz) & \text{ if }~ A=0,\\
(1/A) (e^{Az}-1) & \text{ if }~ B=0.
\end{cases}
\end{equation}
It is easy to show that $K_{A,B}(z)$ belongs to the class $\mathcal{C}(A,B)$. The function $K_{A,B}(z)$ play the role of extremal function for many extremal problems in the class $\mathcal{C}(A,B)$.\\

In the present article, our main aim is to find the sharp estimates of the absolute value of Schwarzian derivative and Schwarzian norm for functions in the class $\mathcal{C}(A,B)$.

\section{Main Results}
Let $\mathcal{B}$ be the class of analytic functions $\omega:\mathbb{D}\rightarrow\mathbb{D}$ and $\mathcal{B}_0$ be the class of Schwarz functions $\omega\in\mathcal{B}$ with $\omega(0)=0$.  According to the Schwarz's lemma if a function $\omega\in\mathcal{B}_0$, then $|\omega(z)|\le |z|$ and $|\omega'(0)|\le 1$. In each of these inequalities, the equality occurs if and only if $\omega(z)=e^{i\alpha}z$, $\alpha\in\mathbb{R}$. The Schwarz-Pick lemma, a natural extension of the Schwarz lemma, yields the estimate $|\omega'(z)|\le (1-|\omega(z)|^2)/(1-|z|^2)$, $z\in\mathbb{D}$ when $\omega\in\mathcal{B}$. In 1931, Dieudonn\'{e} \cite{Dieudonne-1931}  was the first to determine the precise range of variability of $\omega'(z_0)$ for a fixed $z_0\in\mathbb{D}$ over the class $\mathcal{B}_0$.

\begin{lem}[Dieudonn\'{e}'s lemma]\cite{Dieudonne-1931, Duren-1983}
Let $\omega\in\mathcal{B}_0$ and $z_0\ne 0$ be a fixed point in $\mathbb{D}$. The region of variability of $\omega'(z_0)$ is given by
\begin{equation}\label{r-000001}
\left|\omega'(z_0)-\frac{\omega(z_0)}{z_0}\right|\le \frac{|z_0|^2-|\omega(z_0)|^2}{|z_0|(1-|z_0|^2)}.
\end{equation}
Moreover, the equality occurs in \eqref{r-000001} if and only if $\omega\in\mathcal{B}_0$ is a Blaschke product of degree $2$.
\end{lem}

The Dieudonn\'{e}'s lemma is an extension of the Schwarz's lemma as well as Schwarz-Pick lemma. Here, we note that a Blaschke product of degree $n\in\mathbb{N}$ is of the form
$$B(z)=e^{i\theta}\prod_{j=1}^{n}\frac{z-z_j}{1-\bar{z_j}z},\quad z,z_j\in\mathbb{D}, ~\theta\in\mathbb{R}.$$
The Dieudonn\'{e}'s lemma will be key in proving our main results.\\

Before we state our main results, we introduce some sets which we will use throughout our next discussion. Let $E:=\{(A,B):-1\le B<A\le 1\}$ and we consider following subsets of $E$.
\begin{align}\label{r-000005}
\begin{cases}
E_1&:=\{(A,B)\in E:1-\sqrt{1-B^2}< |A+B|< 1+\sqrt{1-B^2}\},\\
E_2&:=\{(A,B)\in E\cup E_1^{c}:|A+B|\le |B|\}\\
&~=\{(A,B)\in E:|A+B|\le 1-\sqrt{1-B^2},~|A+B|\le |B|\},\\
E_3&:=\{(A,B)\in E\cup E_1^{c}:|A+B|> |B|\}\\
&~=\{(A,B)\in E:|A+B|\ge 1+\sqrt{1-B^2},~|A+B|> |B|\}.
\end{cases}
\end{align}
It evident that $E_1, E_2$ and $E_3$ are mutually disjoint and $E=E_1\cup E_2\cup E_3$.

\begin{thm}\label{Thm-r-00005}
For $-1\le B<A\le 1$, let $ f \in\mathcal{C}(A,B)$ be of the form \eqref{r-00001} and $E_1$, $E_2$, $E_3$ are given by \eqref{r-000005}. Then the Schwarzian derivative $S_f(z)$ satisfies the inequality
\begin{equation}\label{r-000010}
|S_f(z)|\le \dfrac{(A-B)(2-|A+B|(1-|z|^2))}{(1-|z|^2)(2-|A+B|(1-|z|^2)-2B^2|z|^2))},\quad z\in\mathbb{D}
\end{equation}
for $(A,B)\in E_1\cup E_2$ and
\begin{equation}\label{r-000015}
|S_f(z)|\le
\begin{cases}
\dfrac{(A-B)(2-|A+B|(1-|z|^2))}{(1-|z|^2)(2-|A+B|(1-|z|^2)-2B^2|z|^2))} & \text{ if }~ z\in S\cap\mathbb{D}, \\[6mm]
\dfrac{|A^2-B^2|}{2(1-|B||z|)^2} & \text{ if }~ z\in S^c\cap\mathbb{D}
\end{cases}
\end{equation}
for $(A,B)\in E_3$, where $S=\{z\in\mathbb{C}: |z|< \delta_1 ~\text{or,}~ |z|> \delta_2\}$ with
$$\delta_1=\frac{|B|-\sqrt{B^2-|A+B|(2-|A+B|)}}{|A+B|}$$
and
$$\delta_2=\frac{|B|+\sqrt{B^2-|A+B|(2-|A+B|)}}{|A+B|}.$$
\end{thm}

\begin{proof}
For $-1\le B<A\le 1$, let $ f \in\mathcal{C}(A,B)$ be of the form \eqref{r-00001}. Then we have
$$1+\frac{zf''(z)}{f'(z)}\prec \frac{1+Az}{1+Bz}.$$
Thus, there exists an analytic function $\omega:\mathbb{D}\rightarrow\mathbb{D}$ with $\omega(0)=0$ such that
$$1+\frac{zf''(z)}{f'(z)}= \frac{1+A\omega(z)}{1+B\omega(z)}.$$
A simple computation gives
$$\frac{f''(z)}{f'(z)}=\frac{(A-B)\omega(z)}{z(1+B\omega(z))},$$
and consequently,
\begin{align}\label{r-000017}
S_f(z)&=\left[\frac{f''(z)}{f'(z)}\right]^{'}-\frac{1}{2}\left[\frac{f''(z)}{f'(z)}\right]^2\\
&=(A-B)\left[\frac{\omega'(z)}{z(1+B\omega(z))^2}-\frac{2\omega(z)+(A+B)\omega^2(z)}{2z^2(1+B\omega(z))^2}\right].\nonumber
\end{align}
Let us consider the transformation $\displaystyle\zeta(z)=\omega'(z)-\frac{\omega(z)}{z}$. By Dieudonn\'{e}'s lemma, the function $\zeta$ varies over the closed disk
$$|\zeta(z)|\le \frac{|z|^2-|\omega(z)|^2}{|z|(1-|z|^2)},$$
for fixed $|z|<1$ and $z\ne 0$. Using the transformation of $\zeta$ in \eqref{r-000017}, we obtain
$$S_f(z)=(A-B)\left[-\frac{(A+B)\omega^2(z)}{2z^2(1+B\omega(z))^2}+\frac{\zeta(z)}{z(1+B\omega(z))^2}\right].$$
Thus,
\begin{align*}
|S_f(z)|&\le (A-B)\left[\frac{|A+B||\omega(z)|^2}{2|z|^2|1+B\omega(z)|^2}+\frac{|\zeta(z)|}{|z||1+B\omega(z)|^2}\right]\\
&\le(A-B)\left[\frac{|A+B||\omega(z)|^2}{2|z|^2(1-|B||\omega(z)|)^2}+\frac{|z|^2-|\omega^2(z)|}{|z|^2(1-|z|^2)(1-|B||\omega(z)|)^2}\right].
\end{align*}
For $0\le s:=|\omega(z)|\le |z|<1$, we have
\begin{align}\label{r-000020}
|S_f(z)|&\le (A-B)\left[\frac{|A+B|s^2}{2|z|^2(1-|B|s)^2}+\frac{|z|^2-s^2}{|z|^2(1-|z|^2)(1-|B|s)^2}\right]\\
&=(A-B)\frac{2|z|^2-s^2(2-|A+B|(1-|z|^2))}{2|z|^2(1-|z|^2)(1-|B|s)^2}\nonumber\\
&=(A-B)g(s),\nonumber
\end{align}
where
\begin{equation}\label{r-000021}
g(s)=\frac{2|z|^2-s^2(2-|A+B|(1-|z|^2))}{2|z|^2(1-|z|^2)(1-|B|s)^2},\quad 0\le s\le |z|<1.
\end{equation}
Now, we wish to find the maximum value of $g(s)$ over the closed interval $[0,|z|]$. To do this, we first find the critical points of $g(s)$ in $(0,|z|)$. A simple computation gives

\begin{equation*}
g'(s)=\frac{2|B||z|^2-s(2-|A+B|(1-|z|^2))}{|z|^2(1-|z|^2)(1-|B|s)^3},
\end{equation*}
and so $g'(s)=0$ yields
\begin{equation}\label{r-000023}
s=\frac{2|B||z|^2}{2-|A+B|(1-|z|^2)}=:s_0(|z|).
\end{equation}

Now, we have to check for what values of $A$ and $B$, the point $s=s_0(|z|)$ lies  in $(0,|z|)$. We first note that the inequality
\begin{align}\label{r-000025}
 s_0(|z|)=\frac{2|B||z|^2}{2-|A+B|(1-|z|^2)}< |z|,
\end{align}
holds true if and only if $k(|z|)>0$, where
\begin{equation}\label{r-000027}
k(|z|)=2-|A+B|-2|B||z|+|A+B||z|^2.
\end{equation}
We also note that $k(0)=2-|A+B|>0$ and the discriminant of $k(|z|)$ is given by
$$\Delta=4((|A+B|-1)^2+B^2-1).$$

Let $(A,B)\in E_1$. If $B=0$ then $g'(s)<0$ in $(0,|z|)$ and so, $g(s)$ is a decreasing function in $(0,|z|)$. Thus, the maximum of $g(s)$ attains at $s=0$. Thus, for $B=0$, the desired result \eqref{r-000010} follows from \eqref{r-000020} and \eqref{r-000021}. If $B\ne 0$ then $s_0(|z|)>0$ and $\Delta<0$. Hence, in this case $k(|z|)$ has no real zero and so $k(|z|)>0$. Consequently, $s_0(|z|)$ lies  in $(0,|z|)$.\\


Let $(A,B)\in E_2$. Then clearly $B\ne 0$, $s_0(|z|)>0$ and $\Delta \ge 0$. If $A+B=0$, then clearly $k(|z|)>0$ and consequently, $s_0(|z|)$ lies  in $(0,|z|)$. If $A+B\ne 0$ then the zeros of $k(|z|)$ in the real line are given by
\begin{equation}\label{r-000030a}
\delta_1=\frac{|B| - \sqrt{B^2-|A+B|(2-|A+B|)}}{|A+B|}
\end{equation}
and
\begin{equation}\label{r-000030b}
\delta_2=\frac{|B| + \sqrt{B^2-|A+B|(2-|A+B|)}}{|A+B|}.
\end{equation}
For $(A,B)\in E_2$ with $A+B\ne 0$, we have $|B|/|A+B|\ge 1$ and so $\delta_2\not\in(0,1)$. Further, $k(0)>0$ and $k(1)=2(1-|B|)>0$ for $B\ne -1$. Therefore if $B\neq -1$ then $k(|z|)$ has either two zeros or no zero in $(0,1)$. Consequently, for $B\neq -1$, $k(|z|)$ has no zero in $(0,1)$ and so, $s_0(|z|)$ lies  in $(0,|z|)$. Again, if $B=-1$ then $\delta_1=1\not\in(0,1)$ and
$$
\delta_2=\frac{2}{|A-1|}-1\ge 1.
$$
Hence, for $B= -1$, $k(|z|)$ has no zero in $(0,1)$ and so, $s_0(|z|)$ lies  in $(0,|z|)$.\\

Therefore, $s_0(|z|)$ lies  in $(0,|z|)$ if $(A,B)\in E_1\cup E_2\setminus\{(A,0)\}$. Since the numerator of $g'(s)$ is a linear function of $s$, and $g'(0)=2|B|/(1-|z|^2)>0$, $g'(s_0(|z|))=0$, it follows that $g'(|z|)<0$. Hence the function $g(s)$ is increasing in $(0, s_0(|z|))$ and decreasing in $(s_0(|z|),|z|)$ and consequently, the maximum of $g(s)$ attains at $s_0(|z|)$. Thus the desired result \eqref{r-000010} follows from \eqref{r-000020} and \eqref{r-000021}.\\

Let $(A,B)\in E_3$. Again, if $B=-1$ then $k(|z|)$ have two zeros $\delta_1=\frac{2}{|A-1|}-1\in(0,1)$ and $\delta_2=1\not\in(0,1)$.
If $B\ne-1$ then the function $k(|z|)$ have two zeros $\delta_1$ and $\delta_2$ in $(0,1)$ which are given by \eqref{r-000030a} and \eqref{r-000030b}, respectively. In any case, by Rolle's theorem, the function $k'(|z|)$ has exactly one zero, say $\alpha$, in $(\delta_1,\delta_2)$. Since, $k'(0)=-2|B|<0$ and $k'(1)=2(|A+B|-|B|)>0$, the function $k(|z|)$ is strictly decreasing in $(0,\alpha)$ and strictly increasing in $(\alpha,1)$. Thus, we conclude that $k(|z|)> 0$ if $z\in S\cap\mathbb{D}$, where $S=\{z\in\mathbb{C}: |z|< \delta_1 ~\text{or,}~ |z|> \delta_2\}$; and $k(|z|)< 0$ if $\delta_1< |z|< \delta_2$. Therefore, $s_0(|z|)$ lies  in $(0,|z|)$ if $z\in S\cap\mathbb{D}$, and $s_0(|z|)$ does not lie in $(0,|z|)$ if $z\in S^c\cap\mathbb{D}$.\\

Therefore, for $z\in S\cap\mathbb{D}$, following the same argument as before, the function $g(s)$ is increasing in $(0, s_0(|z|))$ and decreasing in $(s_0(|z|),|z|)$ and consequently, the maximum of $g(s)$ attains at $s_0(|z|)$. Thus the desired result \eqref{r-000015} follows from \eqref{r-000020} and \eqref{r-000021}. Again, for $z\in S^c\cap\mathbb{D}$, from \eqref{r-000021}, we have
\begin{align*}
g(|z|)= \frac{|A+B|}{2(1-|B||z|)^2}= \frac{|A+B|(1-|z|^2)}{2(1-|B||z|)^2}g(0)\ge g(0)
\end{align*}
as $k(|z|)\le 0$ for $z\in S^c\cap\mathbb{D}$. Thus, the desired result \eqref{r-000015} follows immediately. This completes the proof.

\end{proof}

Before we proceed further, let us discuss the sharpness of the estimate of $|S_f(z)|$ obtained in Theorem \ref{Thm-r-00005}.\\

Let $(A,B)\in E_3$. We consider the function $K_{A,B}(z)\in \mathcal{C}(A,B)$ defined by \eqref{r-00010}. The Schwarzian derivative of $K_{A,B}$ is given by
$$S_{K_{A,B}}(z)=-\frac{A^2-B^2}{2(1+Bz)^2}.$$
For $B>0$ and any $z_0\in S^c\cap\mathbb{D}$ with $-1<z_0\le 0$, we have $|S_{K_{A,B}}(z_0)|=|A^2-B^2|/2(1-|B||z_0|)^2$. Again, for $B<0$ and any $z_0\in S^c\cap\mathbb{D}$ with $0\le z_0< 1$, we have $|S_{K_{A,B}}(z_0)|=|A^2-B^2|/2(1-|B||z_0|)^2$. This shows that the second inequality in \eqref{r-000015} is sharp in such cases.\\

Next, suppose that $A, B$ and $z_0$ satisfy either of the following two conditions:
\begin{enumerate}
\item [\textbf{(C1)}] \quad $(A,B)\in E_1\cup E_2$ and $-1<z_0<1$,
\item [\textbf{(C2)}] \quad $(A,B)\in E_3$ and $-1<z_0<1$ with $z_0\in S\cap\mathbb{D}$.
\end{enumerate}
Depending on $A$ and $B$, we choose a pair of unimodular real numbers $(p,q)$ as follows:
\begin{align}\label{r-000033}
\begin{cases}
(p,q)=(1,1) & \text{when}~ A+B\le 0,\\
(p,q)=(-1,1) & \text{when}~ A+B>0, B\ge 0,\\
(p,q)=(-1,-1) & \text{when}~ A+B>0, B<0.
\end{cases}
\end{align}
For given $A, B$ and $z_0$ satisfying \textbf{(C1)} or \textbf{(C2)}, we choose $(p,q)$ as above and consider the function $f_{z_0,p,q}$ defined by
\begin{equation}\label{r-000035}
1+\frac{zf_{z_0,p,q}''(z)}{f_{z_0,p,q}'(z)}= \frac{1+A\phi(z)}{1+B\phi(z)},
\end{equation}
where
$$\phi(z)=\frac{pz(z-b)}{1-bz},$$
and $b$ is a solution of the equation
\begin{align}\label{r-000040}
&\frac{z_0(z_0-b)}{1-bz_0} =q s_0(|z_0|)=\frac{2q|B|z_0^2}{2-|A+B|(1-z_0^2)},\nonumber\\
&i.e.,~ b=\frac{z_0(2-2q|B|-|A+B|(1-z_0^2))}{2-2q|B|z_0^2-|A+B|(1-z_0^2)}.
\end{align}
We know that whenever $A, B$ and $z_0$ satisfy \textbf{(C1)} or \textbf{(C2)}, the point $s_0(|z_0|)$ lies in $[0,|z_0|)$, where $s_0(|z|)$ is given by \eqref{r-000023}. This ensures that $b\in (-1,1)$ and $\phi$ is a Blaschke product of degree $2$ with $\phi(0)=0$. Hence, the function $f_{z_0,p,q}$ belong to the class $\mathcal{C}(A,B)$. For the function $f_{z_0,p,q}$, the Schwarzian derivative is given by
\begin{align*}
S_{f_{z_0,p,q}}(z)&=\left[\frac{f_{z_0,p,q}''(z)}{f_{z_0,p,q}'(z)}\right]^{'}-\frac{1}{2}\left[\frac{f_{z_0,p,q}''(z)}{f_{z_0,p,q}'(z)}\right]^2\\
&=(A-B)\left[-\frac{(A+B)\phi^2(z)}{2z^2(1+B\phi(z))^2}+\frac{\phi'(z)-\frac{\phi(z)}{z}}{z(1+B\phi(z))^2}\right]\\
&=(A-B)\frac{-(A+B)(z-b)^2+2p(1-b^2)}{2(1-bz+Bpz(z-b))^2}.\\
\end{align*}
Substituting the value of $b$, given in \eqref{r-000040}, and then evaluating the Schwarzian derivative $S_{f_{z_0,p,q}}(z)$ at $z_0$, we obtain
$$S_{f_{z_0,p,q}}(z_0)=-(A-B)\frac{2q^2|B|^2z_0^2(2p+(A+B)(1-z_0^2))-p(2-|A+B|(1-z_0^2))^2}{(1-z_0^2)(2-|A+B|(1-z_0^2)+2pqB|B|z_0^2)^2}.$$
Therefore, for any pair of $(p,q)$, given in \eqref{r-000033}, we have
\begin{equation}\label{r-000043}
|S_{f_{z_0,p,q}}(z_0)|=\frac{(A-B)(2-|A+B|(1-|z_0|^2))}{(1-|z_0|^2)(2-|A+B|(1-|z_0|^2)-2B^2|z_0|^2)}.
\end{equation}
This shows that the inequality \eqref{r-000010} and the first inequality in \eqref{r-000015} are sharp for real $z$.\\

The above discussion shows that the estimate of the Schwarzian derivative $|S_f(z)|$ obtained in Theorem \ref{Thm-r-00005} is sharp for certain real values of $z$. This also helps us to obtain the sharp estimate of the Schwarzian norm $||S_f(z)||$ for functions in $\mathcal{C}(A,B)$ which is given below.

\begin{thm}\label{Thm-r-00010}
For $-1\leq B<A\leq 1$, let $f \in\mathcal{C}(A,B)$ be of the form \eqref{r-00001} and $E_1$, $E_2$, $E_3$ are given by \eqref{r-000005}.
\begin{itemize}
\item[(i)] If $(A,B)\in E_1\cup E_2$, then
\begin{equation}\label{r-000045}
||S_f||\le \begin{cases}
2, & \text{ for }~  B=-1,\\[2mm]
A-B, & \text{ for }~ B\ne -1,~|A+B|\le 2(1-B^2),\\[2mm]
(A-B)\gamma(\alpha), & \text{ for }~B\ne -1,~ |A+B|>2(1-B^2),
\end{cases}
\end{equation}
where $\gamma$ is given by
\begin{equation}\label{r-000040a}
\gamma(t)=\frac{(1-t^2)(2-|A+B|(1-t^2))}{2-|A+B|(1-t^2)-2B^2t^2},
\end{equation}
and $\alpha$ is the unique root in $(0,1)$ of the equation $h(t)=0$ with
\begin{align}\label{r-000040b}
h(t)=&(2-|A+B|)(|A+B|+2B^2-2)-2|A+B|(2-|A+B|)t^2\\
&+|A+B|(2B^2-|A+B|)t^4.\nonumber
\end{align}

\item[(ii)] If $(A,B)\in E_3$, then
\begin{equation}\label{r-000050}
||S_f||\le \frac{2|A^2-B^2|(1-\sqrt{1-B^2})^2}{B^4}.
\end{equation}
\end{itemize}
Moreover, all the estimates are sharp.
\end{thm}
\begin{proof}
For $-1\leq B<A\leq 1$, let $ f \in\mathcal{C}(A,B)$ be of the form \eqref{r-00001}. We prove the theorem by considering two different cases. \\

\noindent\textbf{Case-1: } Let $(A,B)\in E_1\cup E_2$. From \eqref{r-000010}, we obtain
$$|S_{f}(z)|\le (A-B)\frac{2-|A+B|(1-|z|^2)}{(1-|z|^2)(2-|A+B|(1-|z|^2)-2B^2|z|^2)},$$
and hence,
\begin{align}\label{r-000055}
||S_f||&=\sup\limits_{z\in\mathbb{D}}(1-|z|^2)^2|S_f(z)|\\
&\le (A-B)\sup\limits_{0\le |z|<1}\frac{(1-|z|^2)(2-|A+B|(1-|z|^2))}{2-|A+B|(1-|z|^2)-2B^2|z|^2}\nonumber\\
&=(A-B)\sup\limits_{0\le t<1}\gamma(t)\nonumber,
\end{align}
where $\gamma(t)$ be given by \eqref{r-000040a}. To find the supremum of $\gamma(t)$ on $[0,1)$, we consider two different subcases.\\

\noindent \textbf{Subcase-1a:} Let $B=-1$. Clearly, $A\ge 0$. Then $\gamma$ reduces to
$$\gamma(t)=\frac{1+A+(1-A)t^2}{1+A}.$$
It is easy to show that $\gamma$ is a strictly increasing function in $(0,1)$. Hence, from \eqref{r-000055}, we have $||S_f||\le 2$.\\

To show that the estimate is best possible, we consider the function $f_{z_0,1,1}\in\mathcal{C}(A,B)$ defined by \eqref{r-000035} where $-1<z_0<1$. From \eqref{r-000043}, we have
$$(1-z_0^2)^2|S_{f_{z_0,1,1}}(z_0)|=1+A+(1-A)z_0^2\rightarrow 2\quad \text{as}~z_0\rightarrow 1^{-}.$$

\noindent \textbf{Subcase-1b:} Let $B\ne -1$. If $A+B=0$. Then $\gamma$ reduces to
$$\gamma(t)=\frac{1-t^2}{1-B^2t^2}.$$
Then $\gamma$ is a strictly decreasing function in $(0,1)$. Hence, from \eqref{r-000055}, we have $||S_f||\le A-B$.\\

If $A+B\neq 0$, a simple computation gives
 $$\gamma'(t)=\frac{2th(t)}{{(2-|A+B|(1-t^2)-2B^2t^2)}^2},$$
where $h(t)$ is given by \eqref{r-000040b}. Moreover,
\begin{align}\label{r-000056}
 h'(t)&=-4|A+B|(2-|A+B|)t+4|A+B|(2B^2-|A+B|)t^3\\
 &=-4t|A+B|\left(2(1-B^2t^2)-|A+B|(1-t^2)\right).\nonumber
\end{align}
Note that
$$ 2(1-B^2t^2)-|A+B|(1-t^2)\ge (1-t^2)(2-|A+B|)>0.$$
Thus, the function $h$ is strictly decreasing in $(0,1)$. Since the polynomial $h(t)$ is symmetric about the origin and of degree $4$, it follows that $h(t)$ have at most two positive real root. Further, $h(0)=(2-|A+B|)(|A+B|-2(1-B^2))$ and $h(1)=-4(1-B^2)<0$. This shows that $h$ has the unique zero, say $\alpha$, in $(0,1)$ if $|A+B|>2(1-B^2)$ and has no zero in $(0,1)$ if $|A+B|\le 2(1-B^2)$. This further yield that $\gamma(t)$ has maximum at $\alpha$ if $|A+B|>2(1-B^2)$ and $\gamma(t)$ has maximum at $0$ if $|A+B|\le 2(1-B^2)$. Consequently, the required result \eqref{r-000045}  follows from \eqref{r-000055}.\\

To show that the estimate $A-B$ is sharp when $B\ne -1$ and $|A+B|\le 2(1-B^2)$, we consider the function $f_0(z)$ defined by
\begin{equation}\label{r-000057}
1+\frac{zf_0''(z)}{f_0'(z)}=\frac{1+Az^2}{1+Bz^2}.
\end{equation}
Then $f_0\in\mathcal{C}(A,B)$ and the Schwarzian derivative of $f_0$ is given by
$$S_{f_0}(z)=\frac{(A-B)(2-(A+B)z^2)}{2(1+Bz^2)^2}.$$
Moreover, $S_{f_0}(0)=A-B$ and therefore, $||S_{f_0}||=A-B$.\\

To show that the estimate $(A-B)\gamma(\alpha)$ is sharp when $B\ne -1$ and $|A+B|>2(1-B^2)$, we consider the function $f_{\alpha,p,q}$ defined by \eqref{r-000035}, where $\alpha$ is the unique root in $(0,1)$ of the equation $h(t)=0$, where $h(t)$ is given by \eqref{r-000040b}. From \eqref{r-000043}, we have
$$(1-\alpha^2)^2|S_{f_{\alpha,p,q}}(\alpha)|=\frac{(A-B)(1-\alpha^2)(2-|A+B|(1-\alpha^2))}{2-|A+B|(1-\alpha^2)-2B^2\alpha^2}= (A-B)\gamma(\alpha),$$
where $\gamma$ is given by \eqref{r-000040a}. Consequently, $||S_{f_{\alpha,p,q}}(z)||=(A-B)\gamma(\alpha)$.\\

 \noindent\textbf{Case-2: } Let $(A,B)\in E_3$. From \eqref{r-000015}, we obtain
\begin{align*}
|S_f(z)|\le
\begin{cases}
\dfrac{(A-B)(2-|A+B|(1-|z|^2))}{(1-|z|^2)(2-|A+B|(1-|z|^2)-2B^2|z|^2))} & \text{ if }~z\in S\cap \mathbb{D}\\[6mm]
\dfrac{|A^2-B^2|}{2(1-|B||z|)^2} & \text{ if }~ z\in S^c\cap \mathbb{D},
\end{cases}
\end{align*}
where $S$, $\delta_1$ and $\delta_2$ are same as in Theorem \ref{Thm-r-00005}. Therefore,
\begin{align}\label{r-000060}
||S_f||=\sup\limits_{z\in\mathbb{D}}(1-|z|^2)^2|S_f(z)|=\max \{M_1,M_2\},
\end{align}
with
\begin{align*}
M_1=(A-B)\sup\limits_{z\in S\cap \mathbb{D}}\frac{(1-|z|^2)(2-|A+B|(1-|z|^2))}{2-|A+B|(1-|z|^2)-2B^2|z|^2)}
= (A-B)\sup\limits_{t\in S\cap (0,1)}\gamma(t),
\end{align*}
and
\begin{align*}
M_2=(A-B)\sup\limits_{z\in S^c\cap \mathbb{D}}\frac{|A+B|(1-|z|^2)^2}{2(1-|B||z|)^2} =(A-B)\sup\limits_{t\in S^c\cap (0,1)}\gamma_1(t),
\end{align*}
where $\gamma(t)$ is given by \eqref{r-000040a} and $\gamma_1(t)$ is given by
$$\gamma_1(t)=\frac{|A+B|(1-t^2)^2}{2(1-|B|t)^2}.$$
Now, we consider two different subcases.\\

\noindent\textbf{Subcase-2a:} Let $B=-1$. Clearly, $A<0$. Thus, $\delta_1=\frac{2}{1-A}-1\in(0,1)$ and $\delta_2=1$. Therefore,
$$M_1=\displaystyle\sup\limits_{0\le t<\delta_1}(1+A+(1-A)t^2)=1+A+(1-A)\delta_1^2=\frac{2(1+A)}{1-A},$$
and
$$M_2=(A+1)\sup\limits_{\delta_1\le t< 1}\frac{(1-A)(1+t)^2}{2}=2(1-A^2).$$
A simple calculation gives $M_1<M_2$. Therefore, $||S_f||=2(1-A^2)$.\\

\noindent\textbf{Subcase-2b:} Let $B\ne -1$. Then $\delta_1$ and $\delta_2$ lies in $(0,1)$. First we find the value of $M_1$. A simple calculation gives
$$\gamma'(t)=\frac{2th(t)}{(2-|A+B|(1-t^2)-2B^2t^2)^2},$$
where $h(t)$ is a polynomial of degree $4$ given by \eqref{r-000040b} and is symmetric about the origin. Since $(A,B)\in E_3$, it follows that
\begin{align*}
h(0)&=(2-|A+B|)(-2+|A+B|+2B^2)\\
&\ge (2-|A+B|)(B^2-(1-B^2)+\sqrt{1-B^2})\\
&>0, \qquad[~\because 1-B^2\in(0,1)]
\end{align*}
and $h(1)=-4(1-B^2)<0$. Thus, $h$ has exactly one zero in $(0,1)$. Since, $\delta_1$ and $\delta_2$ are the zeros of \eqref{r-000027}, it follows that
\begin{align*}
\gamma(\delta_1)&=\frac{(1-\delta_1^2)(2-|A+B|(1-\delta_1^2))}{2-|A+B|(1-\delta_1^2)-2B^2\delta_1^2}
=\frac{2|B|\delta_1(1-\delta_1^2)}{2|B|\delta_1(1-|B|\delta_1)}
=\frac{(1-\delta_1^2)}{1-|B|\delta_1}\\
&=\frac{2}{|A+B|}.
\end{align*}
Similarly, $$\gamma(\delta_2)=\frac{2}{|A+B|}=\gamma(\delta_1).$$
By Rolle's theorem, $\gamma'(t)$ has atleast one zero, say $\delta$, in $(\delta_1,\delta_2)$. Since $h$ has exactly one zero in $(0,1)$, $\gamma'(t)$ has exactly one zero  $\delta$ in $(\delta_1,\delta_2)$. Consequently, $\gamma$ is strictly increasing in $(0,\delta)$ and strictly decreasing in $(\delta,1)$. Therefore,
$$M_1=(A-B)\sup\limits_{t\in S\cap (0,1)}\gamma(t)=(A-B)\gamma(\delta_1)=(A-B)\gamma(\delta_2)=\frac{2(A-B)}{|A+B|}.$$
Next we will find the value of $M_2$. Clearly,
$$\gamma_1'(t)=\frac{|A+B|(1-t^2)\psi(t)}{(1-|B|t)^3},$$
where $\psi(t)=|B|t^2-2t+|B|$. Since, $\psi(0)=|B|>0$ and $\psi(1)=-2(1-|B|)<0$. Thus, $\psi$ has an unique zero, say $\beta$, in $(0,1)$ and $\beta$ is given by
$$\beta=\frac{1-\sqrt{1-B^2}}{|B|}.$$
A simple calculation yields that
\begin{equation}\label{r-000065}
\gamma_1(\delta_1)=\gamma_1(\delta_2)=\frac{2}{|A+B|}=\gamma(\delta_1)=\gamma(\delta_2).
\end{equation}
By Rolle's theorem, $\gamma'_1(t)$ has atleast one zero in $(\delta_1,\delta_2)$. Since, $\psi$ has exactly one zero $\beta$ in $(0,1)$, $\gamma'_1(t)$ has exactly one zero $\beta$ in $(\delta_1,\delta_2)$.  Consequently, $\gamma_1$ is strictly increasing in $(0,\beta)$ and strictly decreasing in $(\beta,1)$. Therefore, $\gamma_1$ has maximum at $\beta$. That is,
$$M_2=(A-B)\sup\limits_{\delta_1\le t\le \delta_2}\gamma_1(t)=(A-B)\gamma_1(\beta)=\frac{2|A^2-B^2|(1-\sqrt{1-B^2})^2}{B^4}.$$
From \eqref{r-000065}, we note that
$$
M_2=(A-B)\gamma_1(\beta) \ge (A-B)\gamma_1(\delta_1)=(A-B)\gamma(\delta_1)=M_1.
$$
Therefore, from \eqref{r-000060}, we get the desired result.\\

To show that the estimate is sharp, let us consider the function $K_{A,B}(z)$ defined by \eqref{r-00010}. The Schwarzian derivative of $K_{A,B}$ is given by
$$S_{K_{A,B}}(z)=-\frac{A^2-B^2}{2(1+Bz)^2},$$
and so
\begin{align*}
||S_{K_{A,B}}||&=\sup\limits_{z\in\mathbb{D}}(1-|z|^2)^2|S_{K_{A,B}}(z)|=\frac{|A^2-B^2|}{2}\sup\limits_{z\in\mathbb{D}}\frac{(1-|z|^2)^2}{|1+Bz|^2}.
\end{align*}
If $B> 0$, then
\begin{align*}
&\frac{|A^2-B^2|}{2} \sup\limits_{0<t<1}(1-t^2)^2|S_{K_{A,B}}(-t)|= (A-B) \sup_{t\in (0,1)} \gamma_1(t)\\
&\quad = M_2=\dfrac{2|A^2-B^2|(1-\sqrt{1-B^2})^2}{B^4}.
\end{align*}
If $B<0$, then
\begin{align*}
&\frac{|A^2-B^2|}{2}\sup\limits_{0<t<1}(1-t^2)^2|S_{K_{A,B}}(t)| = (A-B) \sup_{t\in (0,1)} \gamma_1(t)\\
&\quad =M_2=\dfrac{2|A^2-B^2|(1-\sqrt{1-B^2})^2}{B^4}.
\end{align*}
Therefore,
$$
||S_{K_{A,B}}||=\dfrac{2|A^2-B^2|(1-\sqrt{1-B^2})^2}{B^4}.
$$
This completes the proof.
\end{proof}

For particular values of $A$ and $B$ with $-1\le B<A\le 1$, one can obtain sharp estimates for the Schwarzian norm of functions belonging to several subclasses of $\mathcal{S}$. If we choose $A=1$ and $B=-1$ in Theorem \ref{Thm-r-00010},  then the first inequality in \eqref{r-000045} provides Schwarzian norm  estimate for the class of convex functions, which was first proved by Robertson \cite{Robertson-1969}.

\begin{cor}\label{r-Cor-00001}
Let $f\in\mathcal{C}(1,-1)=:\mathcal{C}$ be of the form \eqref{r-00001}. Then the Schwarzian norm satisfies the sharp inequality
$$||S_f||\le 2.$$
\end{cor}

If we choose $A=1-2\alpha$, $0\le\alpha\le 1$ and $B=-1$ in Theorem \ref{Thm-r-00010} then the first inequality in \eqref{r-000045} and \eqref{r-000050} provides Schwarzian norm estimate for the class of convex functions of order $\alpha$, which was obtained by Suita \cite{Suita-1996}. Also, the inequality \eqref{r-000050} in Theorem \ref{Thm-r-00010} in conjunction with Theorem \ref{Thm-r-0001} gives the quasiconformal extension.

\begin{cor}\label{r-Cor-00005}
Let $f\in\mathcal{C}(1-2\alpha,-1)=:\mathcal{C}(\alpha)$, $0\le\alpha\le 1$ be of the form \eqref{r-00001}. Then the Schwarzian norm satisfies the sharp inequality
$$||S_f||\le
\begin{cases}
2, &~ \text{for}~~0\le\alpha\le 1/2,\\[2mm]
8\alpha(1-\alpha), & ~~\text{for}~1/2\le\alpha\le 1.
\end{cases}
$$
Further, $f$ can be extended to a $4\alpha(1-\alpha)$-quasiconformal mapping of the Riemann sphere $\widehat{\mathbb{C}}$ when $1/2<\alpha\le 1$.
\end{cor}


If we choose $A=\alpha$, $0< \alpha\le 1$ and $B=0$ in Theorem \ref{Thm-r-00010} then the second inequality in \eqref{r-000045} in conjunction with Theorem \ref{Thm-r-0001} gives the following result.

\begin{cor}\label{r-Cor-00010}
Let $f\in\mathcal{C}(\alpha,0)$,  $0< \alpha\le 1$ be of the form \eqref{r-00001}. Then the Schwarzian norm satisfies the sharp inequality $||S_f||\le \alpha$ and equality occurs for the function $f_0$ which is given by \eqref{r-000057}. Further, $f$ can be extended to a $\alpha/2$-quasiconformal mapping of the Riemann sphere $\widehat{\mathbb{C}}$.
\end{cor}

If we choose $A=\alpha$ and $B=-\alpha$, $0< \alpha\le 1$ in Theorem \ref{Thm-r-00010} then the second inequality in \eqref{r-000045} in conjunction with Theorem \ref{Thm-r-0001} gives the following result.

\begin{cor}\label{r-Cor-00015}
Let $f\in\mathcal{C}(\alpha,-\alpha)$,  $0< \alpha\le 1$ be of the form \eqref{r-00001}. Then the Schwarzian norm satisfies the sharp inequality $||S_f||\le 2\alpha$ and equality occurs for the function $f_0$ which is given by \eqref{r-000057}. Further, $f$ can be extended to a $\alpha$-quasiconformal mapping of the Riemann sphere $\widehat{\mathbb{C}}$ when $0< \alpha< 1$.
\end{cor}


%

%

%

\vspace{5mm}
\textbf{Data availability:} Data sharing not applicable to this article as no data sets were generated or analyzed during the current study.\\[2mm]

\textbf{Authors contributions:} All authors contributed equally to the investigation of the problem and the order of the authors is given alphabetically according to their surname. All authors read and approved the final manuscript.\\[2mm]

\textbf{Acknowledgement:} The second named author thanks the University Grants Commission for the financial support through UGC Fellowship (Grant No. MAY2018-429303).

\end{document}